\documentclass[leqno]{article}
\usepackage{amsmath}
\usepackage{amsthm}
\usepackage{amssymb}
\usepackage{amsxtra}     

\usepackage{epsfig}
\usepackage{verbatim}
\usepackage{hyperref}
\usepackage{enumitem}

\usepackage{algorithmicx}
\usepackage{algpseudocode}

\theoremstyle{plain}
\newtheorem{Theorem}{Theorem}

\newtheorem{Corollary}[Theorem]{Corollary}
\newtheorem{Lemma}[Theorem]{Lemma}

\theoremstyle{definition}
\newtheorem{Definition}[Theorem]{Definition}

\newtheorem{Example}[Theorem]{Example}

\theoremstyle{remark}
\newtheorem{Remark}[Theorem]{Remark}
\include{header}

\usepackage{amssymb,amsmath}
\usepackage{tikz}
\usepackage{tikz-cd}
\usepackage{amssymb, amsmath, amsfonts, amsthm, graphics}
\usepackage[all]{xy}
\usepackage{amsrefs}

\usetikzlibrary{arrows,decorations.pathmorphing,decorations.pathreplacing,positioning,shapes.geometric,shapes.misc,decorations.markings,decorations.fractals,calc,patterns}

\newcommand{\into}{\hookrightarrow}

\DeclareMathOperator{\Hom}{Hom}

\DeclareMathOperator{\Obj}{Obj}

\DeclareMathOperator{\spa}{span}

\newcommand{\II}{\mathbb{I}}

\newcommand{\KK}{\mathbb{K}}

\newcommand{\ZZ}{\mathbb{Z}}

\newcommand{\cA}{\mathcal{A}}

\newcommand{\cF}{\mathcal{F}}

\newcommand{\cH}{\mathcal{H}}
\newcommand{\cI}{\mathcal{I}}

\newcommand{\cP}{\mathcal{P}}

\newcommand{\cR}{\mathcal{R}}
\newcommand{\cS}{\mathcal{S}}
\newcommand{\cT}{\mathcal{T}}

\DeclareMathOperator{\Id}{Id}

\newcommand{\R}{\KK[x_1, \ldots, x_n]}
\newcommand{\E}{\we(x_1, \ldots, x_n)}
\newcommand{\x}{\mathbf{x}}

\newcommand{\y}{\mathbf{y}}

\newcommand{\W}{\Omega}

\newcommand{\WVD}{\Omega_{V,d}}

\newcommand{\wv}{W \otimes V}

\DeclareMathOperator{\GL}{GL}
\newcommand{\glv}{\GL(V)}

\newcommand{\blank}{ \rule[0.1cm]{0.3cm}{0.1pt}}
\newcommand{\Blank}{ \rule[0.1cm]{0.4cm}{0.1pt}}

\newcommand{\ten}{\blank \otimes V}
\newcommand{\TV}{\cT_V}

\newcommand{\HVD}{\cH_{V,d}}

\newcommand{\Exterior}{\mathchoice{{\textstyle\bigwedge}}%
    {{\bigwedge}}%
    {{\textstyle\wedge}}%
    {{\scriptstyle\wedge}}}
\newcommand{\we}{\Exterior}

\newcommand{\Sl}{\cS_{\lambda}}
\newcommand{\Slp}{\cS_{\lambda'}}
\newcommand{\Sm}{\cS_{\mu}}
\newcommand{\Smp}{\cS_{\mu'}}
\newcommand{\Sn}{\cS_{\nu}}
\newcommand{\Snp}{\cS_{\nu'}}

\newcommand{\la}{\lambda}

\newcommand{\rep}{\mathbf{Rep}_V}
\newcommand{\ve}{ \mathbf{Vec}}
\newcommand{\poly}{\mathbf{Poly}}
\newcommand{\polyd}{\poly_d}
\newcommand{\gp}{\mathbf{GPoly}}

\DeclareMathOperator{\sym}{Sym}

\newcommand {\id}[1][ ]{\mathbf{1} {#1}}

\DeclareMathOperator{\reg}{reg}

\newcommand {\ia}{\cI_{\cA}}

\usepackage{bm}
\newcommand*{\allbf}[1]{\ifmmode\bm{#1}\else\textbf{#1}\fi}

\title{Degree bounds for invariant skew polynomials}
\author{Francesca Gandini}
\date{}

\begin{document}

\maketitle
\begin{abstract}
 When we consider the action of a finite group on a polynomial ring, a polynomial unchanged by the action is called an invariant polynomial. A famous result of Noether states that in characteristic zero the maximal degree of a minimal invariant polynomial is bounded above by the order of the group. Our work establishes that the same bound holds for invariant skew polynomials in the exterior algebra. Our approach to the problem relies on a theorem of Derksen that connects invariant theory to the study of ideals of subspace arrangements. We adapt his proof over the polynomial ring to the exterior algebra, reducing the question to establishing a bound on the Castelnuovo-Mumford regularity of intersections of linear ideals in the exterior algebra. We prove the required regularity bound using tools from representation theory. In particular, the proof relies on the existence of a functor on the category of polynomial functors that translates resolutions of ideals of subspace arrangements over the polynomial ring to resolutions of ideals of subspace arrangements over the exterior algebra.   
\end{abstract}

\section{Introduction}

In this article we study invariant skew polynomials in the exterior algebra.  A central goal of invariant theory is to determine a minimal set of generating invariants (i.e., invariant polynomials) for the action of a group on the set of variables. Here we adapt an algorithm of Derksen \cites{inv,ci} for finding invariants over the polynomial ring  to an equivalent procedure in the exterior algebra. Even though the algorithm is explicit, it is often not practical due to the number of variables involved. 

When explicitly finding invariants is computationally expensive, We want to investigate another central question in invariant theory: determining the smallest integer $\beta$ such that a
set of generating invariants of degree $\leq \beta$ exists.  Emmy Noether \cite{em} proved that for invariant polynomials under the action of a finite group $G$ of order $|G|$ we have that $\beta(G) \leq |G|$, assuming that the ground field has characteristic zero (or the characteristic of the field is much larger than the order of the group). This famous result is known as \emph{Noether's Degree Bound} . In this paper we establish that Noether's Degree Bound holds in the exterior algebra in characteristic zero.

Many mathematicians have worked on determining if analogs of Noether's Degree Bound hold in more general settings. The bound does not hold in general in positive characteristic, so researchers focused on proving the bound when the characteristic does not divide the order of the group (non-modular case). To extend the result to the non-modular case, in \cite{inv} Derksen showed that generating invariant polynomials can be computed from a set of polynomials vanishing on a certain subspace arrangement. The non-modular case of Noether's Degreee Bound was then proved independently by Fleishman \cite{fl} and Fogarty \cite{fo}, before Derksen and Sidman \cite{ds1} proved the regularity bound needed for the subspace arrangement approach to become a complete proof. We extend here the connection between invariant theory and the study of subspace arrangement to invariant skew polynomials in the exterior algebra.

On the one hand, it might not be surprising that bounds over the symmetric algebra transfer to bounds over the exterior algebra. On the other hand, invariant theory over non-commutative rings has unexpected behaviors. Kirkman, Kuzmanovich,  and Zhang \cite{kirk}, showed that invariants in the skew polynomials algebra (of which the exterior algebra is a quotient) do not satisfy Noether's Degree Bound. Research in non-commutative invariant theory is active and currently focusing on the action of Hopf algebras as an analog of group actions, see the survey article by Kirkman \cite{survey}. A very recent extension of several results from commutative invariant theory to the non-commutative setting is in the work of Kirkman, Won, and Zhang \cite{kwz}.

Suppose that $W_1,W_2,\dots,W_t$ are subspaces of an $n$-dimensional $\KK$-vector space $W\cong \KK^n$ and let $I_1,I_2,\dots,I_t\subseteq \KK[x_1,x_2,\dots,x_n]$ be the vanishing ideals of $W_1,W_2,\dots,W_t$. These vanishing ideals are linear ideals in the sense that they are generated by linear forms. To use the subspace arrangement approach for invariant theory, Derksen initially conjectured in \cite{inv} that the vanishing ideal of a union of $t$ subspaces is generated by polynomials of degree at most $t$. He then proved that Noether's Degree Bound holds in the non-modular case if the conjecture holds for $t=|G|$. Sturmfels made an even stronger conjecture: the vanishing ideal of a union of $t$ subspaces has Castelnuovo-Mumford regularity at most $t$. Because the regularity of an ideal bounds the degree of its generators, a regularity bound yields a degree bound for the generators, proving Derksen's initial conjecture from \cite{inv}. Conca and Herzog had shown that the Castelnuovo-Mumford regularity of the product ideal $I_1I_2\cdots I_t$ is equal to $t$ (\cite{ch}). Derksen and Sidman proved Sturmfels' conjecture, namely they showed that the Castelnuovo-Mumford regularity of the intersection ideal $I_1\cap I_2\cap \cdots\cap I_t$ is at most $t$ (\cite{ds1}); similar results hold for more general ideals constructed from linear ideals (\cite{ds2}). 

In this paper, we will use the regularity bound on the intersection of $t$ linear ideals to derive a result in non-commutative invariant theory. In particular, we generalize Derksen's subspace approach to the exterior algebra and prove the following the result.
\begin{Theorem}
In characteristic zero, for the action of a finite group $G$ on the exterior algebra $\we(V)$, we have that Noether's Degree Bound holds, meaning that  $\beta_V(G) \leq |G|$, for every finite dimensional vector space $V$.
\end{Theorem}
We establish the theorem by considering a subspace arrangement of cardinality $t= |G|$ associated to the group $G$. The regularity bound on the ideal of the subspace arrangement in the exterior algebra from ~\cites{article} gives us the degree bound on a minimal set of generating invariants. The idea of using polynomial functors to establish results in invariant theory goes back to the times of Weyl and it continues nowadays in our work, the work of Derksen and Makam \cite{dm}, and the work of Snowden \cite{sno}.

 The regularity bound from our previous work ~\cite{article} leverages regularity results for the symmetric algebra $ \sym(W^*)$ to prove similar
 regularity bounds for ideals of subspace arrangements over the exterior algebra $ \we(W^*)$. To this goal, we associate to a subspace arrangement a family of ideals that is stable under the action of the general linear group and study them using tools from representation theory. We then construct the transpose functor $\W$ on the category of graded polynomial functors which translates modules over $\sym$ to modules over $\we$. In \cite{sam2} Sam and Snowden had already established this connection and used it to prove regularity results. Furthermore, Snowden used twisted commutative algebras to give bounds to the minimal resolution of invariant rings of finite groups in \cite{sno}. Our construction of the transpose function $\W$ leverages the graded structure of polynomial functors to forgo the need to study the structure theory of $\GL_\infty$-representations.

\section{The subspace arrangement theorem over for exterior algebra}
In ~\cite{thesis} and ~\cite{article}, the author used a functor on polynomial functors, the transpose functor $\Omega$, which preserves regularity. In particular, for a subspace arrangement $\ia$ of cardinality $t$ this implies that $\Omega(\ia)$ is $t$-regular. Before we apply this result to invariant skew polynomials, we establish some notation and definitions relevant to our use of polynomial functors.

Throughout this paper we will consider a field $\KK$ of characteristic 0. Our treatment of polynomials functors follows the classical treatment of MacDonald \cite{mac}. 
\begin{Definition}
A functor $\cF$ from $\ve$ to $\ve$ is a \emph{polynomial functor} if  the map 
$$\cF:\Hom(X,Y) \to \Hom(\cF(X),\cF(Y))$$ is a polynomial mapping for all finite dimensional $\KK$-vector spaces $X,Y$. We say that $\cF$ is homogeneous of degree $d$ if $\cF(\lambda h)=\lambda^d \cF(h)$ for every linear map $h \in \Hom(X,Y)$ and every scalar $\lambda\in \KK$.
\end{Definition}

Let $\cF$ be a polynomial functor. We will work in the graded category of polynomial functors $\gp$, where we allow for infinite direct sums of polynomial functors, but we imposes that in each degree $d$ we have a finite sum of homogeneous polynomial functors of degree $d$. For more details on the category $\gp$ see ~\cite[Section 2]{article}.

Every polynomial functor is naturally equivalent to a finite direct sum of Schur functors, denoted $\cS_\lambda$ and indexed by integer partitions. A partition is a sequence $\lambda=(\lambda_1,\lambda_2,\dots,\lambda_r)$ of positive integers with $\lambda_1\geq \lambda_2\geq \cdots\geq \lambda_r$.
For each partition $\lambda$ one can define a polynomial functor $\cS_\lambda:\ve\to \ve$ that is homogeneous of degree $|\lambda|=\lambda_1+\lambda_2+\cdots+\lambda_r$. For a finite dimensional vector space $V$, the representation $\cS_\lambda (V)$ is an irreducible representation of $\GL(V)$. For example, the space $\cS_{(d)}(V)=\sym^d(V)$ is the $d$-th symmetric power of $V$ whilst the space $\cS_{(1,1,\dots,1)}(V)=\cS_{(1^d)}(V)=\we^d(V)$
 is the $d$-th exterior power of $V$.
   It follows from Schur's lemma that 
 $$\Hom(\cS_\lambda,\cS_\mu)=\begin{cases}
 \KK & \mbox{if $\lambda=\mu$;}\\
 0 & \mbox{if $\lambda\neq \mu$.}
 \end{cases}
 $$
 By grouping the $\cS_\lambda$'s together we see that
 every  polynomial functor $\cP \in \gp$ is naturally equivalent to a direct sum 
 \[ \cP= \bigoplus_d \cP_d \cong \bigoplus_d \bigoplus_{\la \vdash d} \Sl^{a_\la} \cong \bigoplus_d \bigoplus_{\la \vdash d} \Sl \otimes \KK^{a_\la}  , \] 
where $a_\la$ is the multiplicity of the irreducible representation $\Sl$ which can be recorded in the multiplicity space $\KK^{a_\la}$.

On the category $\gp$ we define the transpose functor $\W$ on polynomial functors of degree $d$ as the direct limit of the functors $\WVD$. To construct $\WVD$, we will need the functor  $\TV:\ve\to\rep$, defined as the functor $\ten$ which acts on objects by mapping $W \in \Obj(\ve)$ to $\wv \in \Obj (\rep)$, and the functor $\HVD:\rep\to\ve$, defined as the functor $\Hom_{\glv}(\we^d(V), \Blank)$ which acts on objects by mapping $\glv$-representation $U$ to its $\we^d(V)$-isotopic component.
\begin{Definition}[\cite{article}, Definition 3.1]
Fix a vector space $V$ of dimension $n$. Fix a degree $d$ such that $n \geq d$.
For a polynomial functor $\cP_d$ of degree $d$ in $\gp$, the functor $\WVD(\cP_d):\ve\to\ve$ is defined as 
\[ \WVD(\cP_d) = \HVD\circ \cP_d \circ\TV .\] 
\end{Definition}

To define $\W$, we need to establish a directed system of functors $\WVD$. For $i<j$, let $\rho_{ji}$ be the inclusion $\rho_{ji}: V_i\to V_j$. So we have a natural transformation $g_{i,j}:\Omega_{d,V_j}(\cP_d)\to \Omega_{d,V_i}(\cP_d)$ with a unique splitting $f_{j,i}: \Omega_{d,V_i}(\cP_d)\to \Omega_{d,V_j}(\cP_d)$ such that $g_{i,j}\circ f_{j,i}$ is the identity. 
\begin{Definition}[\cite{article}, Definition 3.6-3.7]
Consider the direct system where $f_i=f_{i+1,i}$, as described above. We define $\W_d(\cP_d)$ to be the following direct limit 
\[ \W_d(\cP_d) = \varinjlim_i \W_{d,V_i}(\cP_d) .\]  
Let $\cP \in \poly$ be decomposed as $\cP = \bigoplus_d \cP_d$. We define $\W(\cP)$ to be: 
\[ \W(\cP) = \bigoplus_d \W_d ( \cP_d) .\] 
\end{Definition}
\begin{Remark}
In concrete terms, we often think of $\W$ in terms of its effect on the simple objects of $\polyd$, the Schur functors $\Sl$. In particular, the polynomial functor $\W (\Sl)$ is naturally equivalent to $\Slp$ and $\W(\sym) = \we$ (see \cite[Section 4.3.1]{thesis}).
\end{Remark}

In the category $\gp$ we can define polynomial functors associated to subspace arrangements.  Let $\cA = \{ W_1, \ldots, W_t \} \subset W \cong \KK^m $ be  a collection of linear subspaces in $W$ and let $J_i = \II(W_i)$ be the their vanishing ideals. The intersection ideal associated to $\cA$ is the ideal $ \II (\cA) = \II (W_1 \cup \cdots \cup W_t ) = \bigcap_{i=1}^t J_i$. 
\begin{Definition}
Let $V$ be any object in the category of finite dimensional vector spaces $\ve$. In the polynomial ring $\cS(W^* \otimes V^*)$ we define $\ia (V)$ to be the vanishing ideal of the subspace arrangement $\cA \otimes V$, i.e.,
\[ \ia(V) = \II( W_1\otimes V \cup \cdots \cup W_t \otimes V). \]
\end{Definition}
Notice that for every vector space $V$, we have that $\ia(V)$ is a homogeneous ideal in $\cS(\wv)$ so that can define a monomorphism $\ia \into \cS(W \otimes \blank)$. Thus for every $d$, we have that $(\ia)_d$ is a polynomial functors of degree $d$, giving $\ia$ the structure of an object in $\gp$. Furthermore, $\ia$ is a module functor over the algebra functor   $\cS(W \otimes \blank)$ in $\gp$.  For details on  the axioms of algebra and module functors see \cite[Section 2.2]{article} and for more details on the description of $\ia$ as a module functor we refer the reader to \cite[Section 5]{article}. 


We want to study the action of the transpose functor $\W$ on the polynomial functor $\ia$ and 
express the result in terms of the intersection ideal in the exterior algebra.
\begin{Definition}
Let $\cA = \{ W_1, \ldots, W_t \} \subset W \cong \KK^m $ be a subspace arrangement. Suppose that $S_i$ be the set of linear forms vanishing on $W_i$ and let $J_i$ be the ideal generated by $S_i$ in $E= \we(W^*)$. We define $\II'(\cA)$, the intersection ideal of $\cA$ in $E$ to be 
\[ \II'(\cA) = \bigcap_i J_i . \]
\end{Definition}
\begin{Remark}
A word of caution: we cannot consider the non-linear skew polynomials in $\we(W^*)$ as $\KK$-valued functions on $W$. The reader may think of $\II'(\cA)$ as the vanishing ideal of $\cA$ in $E$, but one must be careful not to abuse this heuristic.  
\end{Remark}

\begin{Theorem}[Subspace arrangement theorem for the exterior algebra]\label{subext}
If $\cA$ is an arrangement of $t$ subspaces in $\KK^n$, then the ideal $\II'(\cA)$ in the exterior algebra $\we (x_1, \ldots, x_n )$ is $t$-regular. In particular, this ideal is generated in degree at most $t$. 
\end{Theorem}
\begin{proof}
Consider the polynomial functor $\ia'$ associated to the intersection ideal $\II'(\cA)$ in the exterior algebra. We have that $\ia'= \Omega(\ia)$ for $\ia$ the polynomial functor associated to the intersection ideal $\II(\cA)$ in the symmetric algebra. By a result of Derksen and Sidman \cite{ds1}, we know that $\II(\cA)$ is $t$-regular, so it is generated in degree at most $t$. Moreover, the same result also gives us that $\ia(V)$ is $t$-regular, for any finite dimensional vector space $V \in \Obj(\ve)$. Thus, $\ia$ is $t$-regular. Applying $\Omega$, we can conclude that $\ia'=\Omega(\ia)$ is also $t$ regular, meaning that $\ia'(V)$ is $t$-regular for any $V \in \Obj(\ve)$. In particular, $\II'(\cA)$ is $t$-regular, so it is generated in degree at most $t$.
\end{proof}

Derksen in ~\cites{inv, ci} provides an algorithm to compute invariant polynomials for the action of a finite group $G$ acting on a vector space $V$ using the ideal generators of the vanishing ideal of the group arrangement.
\begin{Definition}
Let $G$ be a finite group and $V$ a representation of $G$. We define the group arrangement $ \cA_G$ as
\[ \cA_G = \bigcup_{g \in G} \{ (v, g \cdot v) \mid v \in V \} \subset V \oplus V. \]
\end{Definition}

We can find explicit equations for $\cA$ by picking a basis for $V \oplus V $ with coordinate functions $(x_1, \ldots, x_n, y_1, \ldots, y_n )$. Let $A(g)=(A(g)_{i,j})$ be the matrix representing the action of $g$ on $V$. Then the linear subspace $V_g$ is cut out by the linear equations 
$y_i=\sum_{j=1}^n A(g)_{i,j} \, x_j$, for $i=1,2,\dots,n$. 

We can consider the ideal generated by the set 
$$S_g=\left\{y_1-\sum_{j=1}^n A(g)_{1,j}\, x_j,\ldots,y_n-\sum_{j=1}^n A(g)_{n,j} \, x_j \right\}$$ 
in the polynomial ring or in the skew polynomial ring. For $J_g = \II(V_g)$ in the polynomial ring $\KK[\x,\y]$, we say that $J_g$ is the ideal of functions vanishing on $V_g$. For $f \in J_g'$, the ideal generated by $S_g$ in the exterior algebra $\we(\x,\y)$, we have that the ring homomorphism $\phi_g$ given by the substitution $y_i \mapsto \sum_{j=1}^n A(g)_{i,j} x_j$ sends $f$ to zero i.e., $\phi_g(f) =0$.  In this sense we mean that $f$ ``vanishes" on $V_g$. Moreover, the vanishing ideal of $\cA_G$ in the polynomial ring is the intersection of the ideals $J_g = \II(V_g)$. The ideal $\II'(\cA_G)$ is the intersection of the ideals $J'_g= \langle S_g \rangle$ in $E = \we( \x, \y)$, so it is the intersection of the ideals generated generated by each set $S_g$ in the exterior algebra. Notice that for $f \in \II'(\cA)$ we have that the substitution $\phi_g$ is such that $\phi_g(f)=0$, for all $g$. In this sense we say that $\II'(\cA)$ is the ``vanishing" ideal of $\cA_G$ in $E= \we(\x,\y)$.
    
In the next sections we will prove an analog of Derksen's result for computing invariants in the new context of invariant skew polynomials in the exterior algebra. We will first prove a classical result of invariant theory in the general setting of (possibly non-commutative) graded algebras. 

\section{Hilbert invariant theorem for graded algebras}
Let $A$ be a graded $\KK$ algebra such that $A = \bigoplus_{d \geq 0} A_d$ and $A_0 \cong \KK$. Let $G$ be a linearly reductive group acting regularly on $A$ by degree-preserving automorphisms. In particular, notice that $A_d$ is a representation of $G$ for every $d$. We 
denote the subspace of fixed points of $G$ in $A_d$ by 
\[ A_d^G= \left\{ f \in A_d \mid  g \cdot f = f \, , \, \forall g \in G \right\}. \]
We denote by $A_+$ the ideal $\bigoplus_{d > 0} A_d$ whilst $A_+^G$ denotes $\bigoplus_{d > 0} A_d^G$. Notice that for $A^G = \bigoplus_{d \geq 0} A^G_d$, we have that $A^G \cong \KK \oplus A_+^G$ as $g \cdot a_0 = a_0$, for all $a_0 \in A_0 \cong \KK$. Moreover, we let the Hilbert left ideal $I_G$ be
\[ I_G = A A_+^G  = \left\{ \sum a_i f_i \mid a_i \in A, f_i \in A_d^G , \text{ for some $d>0$} \right\}. \] 

We introduce the notion of a Reynolds operator in this general setting. For a linearly reductive group, the space $V^G$ has a unique $G$-stable complement in $V$. This means that there is a unique linear $G$-invariant projection $V\to V^G$.
\begin{Definition}
Let $G$ be a linearly reductive group and $V$ be a representation of $G$. We define $\cR_G$, the Reynolds operator of $G$, to be the unique $G$-equivariant linear projection $\cR_G:V\to V^G$.
\end{Definition}

\begin{Lemma}
For any $G$-representations $V$ and $W$ and $G$-equivariant map $\phi: V \to W$  we have that the following diagram commutes:
\begin{equation}\label{eq:commuting}
\begin{tikzcd}[column sep = 1.2in]
 V \arrow[r, "\phi"] \arrow[d, "\cR_G" left] & W  \arrow[d, "\cR_G"]  \\
 V^G \arrow[r, "\phi\mid_{V^G}"] &  W^G 
\end{tikzcd} 
\end{equation}
\end{Lemma}

\begin{Corollary}
Notice that the operator $\cR_G$ will satisfy the following properties:
\begin{enumerate}[label=(\roman*)]
\item For a $G$-stable subspace $U$, we have that $\cR_G(U)=U^G$;
\item We have that $\cR_G$ is a $A^G$-bimodule homomorphism, meaning that if $a,b \in A^G$, then $\cR_G(afb) = a \cR_G(f) b$, for any $f \in A$. 
\end{enumerate}
\end{Corollary}

\begin{proof} \ 
\begin{enumerate}[label=(\roman*)]
\item For the inclusion map $\phi:U\to V$ we have a commuting diagram (\ref{eq:commuting}),
so $\cR:V\to V^G$ restricts to the Reynolds operator $\cR:U\to U^G$.
\item Define $\phi:A\to A$ by $\phi(f)=afb$. Then $\phi$ is a $G$-equivariant linear map,
hence we have 
$$\cR_G(afb)=\cR_G(\phi(f))=\phi(\cR_G(f))=a\cR_G(f)b.$$
\end{enumerate}
\end{proof}

We will use this general notion of a Reynolds operator applied to the graded algebra $A$ to prove a version of Hilbert invariant theorem for $A$. We build up to this result with the following lemmas.

\begin{Lemma}[Ideal generators are algebra generators]\label{l1}
Suppose that $\{f_1, \ldots, f_r \}$ generate $A_+$ as a left ideal, then $A= \KK \langle f_1, \ldots, f_r \rangle$, the subalgebra generated by $f_1,\dots,f_r$. 
\end{Lemma}

\begin{proof}
As $A$ is a graded algebra, it is enough to prove that for any homogeneous $g \in A$, we have that $g \in \KK \langle f_1, \ldots, f_r \rangle$. We will prove the claim by induction on $d = \deg(g)$. If $\deg(g)= 0$, then the claim is obvious. Assume that $d = \deg(g) > 0$, so that $g \in A_+$. Then we can write $g = \sum a_i f_i$ and after cancellation we can assume that $\deg(a_i) + \deg(f_i) = d$ for all $i$. Thus, $\deg(a_i) <d$ and by induction $a_i \in \KK \langle f_1, \ldots, f_r \rangle$. Therefore, $g \in \KK \langle f_1, \ldots, f_r \rangle$, as required.
\end{proof}

\begin{Lemma}[Invariant ideal generators are invariant algebra generators]\label{l2}
Suppose that homogeneous $f_1, \ldots, f_r \in A^G_+\subset I_G$ generate $I_G$ as a left ideal.
Then $A^G= \KK \langle f_1, \ldots, f_r \rangle$, the free algebra of words in the letters $f_i$. 
\end{Lemma}

\begin{proof}
As $A^G$ is a graded algebra, it is enough to prove that for any homogeneous $g \in A^G$, we have that $g \in \KK \langle f_1, \ldots, f_r \rangle$. If $\deg(g)= 0$, then the claim is obvious. Assume that $d = \deg(g) > 0$, so that $g \in A^G_+ \subset I_G$. Then $g = \sum a_i f_i$. Apply the Reynolds operator $\cR$ to $g$. We have that
\[ g = \cR(g)= \sum \cR(a_i f_i) = \sum \cR(a_i) f_i, \]
by the assumption that $f_i \in A_+^G$ and property (ii) of $\cR$. Thus, $g$ lies in the left ideal generated by $\{f_1, \ldots, f_r\}$ in $A^G$. This means that $\{f_1, \ldots, f_r\}$ is a set of left ideal generators for $A_+^G$ in $A^G$. Therefore, by Lemma \ref{l1}, we have that $A^G = \KK \langle f_1, \ldots, f_r \rangle$, as required. 
\end{proof}

Before stating our final lemma, we need the graded version of Nakayama's lemma.
\begin{Lemma}[Graded Nakayama's lemma]\label{gnaka}
Suppose that $M$ is a finitely generated graded left $A$-module such that $A_+ M = M$. Then we have $M=0$.
\end{Lemma}

\begin{proof}
Let $i$ be the smallest positive degree such that $M_i \neq 0$. Then $M_i \cap A_+M = 0$, but by assumption $A_+M = M$, so $M_i \cap A_+M = M_i$. Therefore, $M = 0$.
\end{proof}

\begin{Lemma}[Replacing ideal generators with invariant ideal generators]\label{l3}
Suppose that $\{f_1, \ldots, f_r \}$ generate $I_G$ as a left ideal. Then $\{\cR(f_1), \ldots, \cR(f_r)\}$ generate $I_G$ as a left ideal.
\end{Lemma}

\begin{proof}
Notice that since $I_G$ is $G$-stable, we have that $\cR(f_i) \in I_G$, for all $i$, by property (i) of $\cR$. Thus, the left ideal generated by $\{\cR(f_1), \ldots, \cR(f_r)\}$ is contained in $I_G$. Let $J_G$ be the left ideal generated by $\{\cR(f_1), \ldots, \cR(f_r)\}$. We have that $J_G \subseteq I_G$. We want to show that this containment is actually an equality.

Consider the space $I_G/ A_+ I_G$. We have that $G$ acts trivially on $I_G/ A_+ I_G$, so that the Reynolds operator is the identity on this space. In particular, $f_i + A_+ I_G = \cR(f_i) + A_+ I_G$, for all $i$. Thus, $I_G = J_G + A_+ I_G$. Consider the module $I_G /J_G$. We have that
\[ A_+ (I_G/J_G) = (J_G + A_+I_G)/J_G = I_G /J_G. \]
By Nakayama's lemma \ref{gnaka}, we have that $I_G /J_G =0$. Therefore $J_G = I_G$, as required.
\end{proof}

\begin{Theorem}[Hilbert invariant theorem]\label{hit}
Suppose that $\{f_1, \ldots, f_r \}$ generate $I_G$ as a left ideal. Then $A^G= \KK \langle f_1, \ldots, f_r \rangle$, the free algebra of words in the letters $f_i$.
\end{Theorem}

\begin{proof}
Notice that by Lemma \ref{l3} we can assume without loss of generality that $f_i \in A_+^G$, for all $i$. Then the result is an immediate consequence of Lemma \ref{l2}.
\end{proof}

\section{Computing invariants over the exterior algebra}

We begin we the definition of the usual terms from invariant theory in the context of the exterior algebra $E = \we (x_1, \ldots, x_n )$. Notice that we can think of the exterior algebra as 
\[ E = \frac{\KK \langle x_1, \ldots, x_n\rangle}{(x_ix_j+x_jx_i, 1\leq i\leq j\leq n)}. \]
 In particular, $E$ is a finite-dimensional graded algebra where $E_d = 0$ for $d >n$. We denote the multiplication in $E$ by $\wedge$. As a $\KK$-vector space, we have that
\[E_d = \spa_\KK \{ x_{i_1} \wedge x_{i_2} \wedge \cdots \wedge x_{i_d}\mid 1\leq i_1<i_2<\cdots<i_d\leq n \}. \]
 Moreover, we notice that every homogeneous ideal in $E$ is a two-sided ideal: for homogeneous $a,f$ we can rewrite any product $a \wedge f$ as $\pm f \wedge a$ using the skew commutative relation $x_ix_j=-x_jx_i$, where the sign $\pm$ is determined by the degrees of $f$ and $a$. In fact, $ a \wedge f = (-1)^{\deg(f)\deg(a)} f \wedge a$. 
 
\begin{Definition}
Let $V$ be a an $n$-dimensional representation of $G$ and consider the invariant skew polynomials for this action in $E = \we(V) = \we (x_1, \ldots, x_n )$. The \emph{Hilbert ideal} $J_G$ of $G$ in $E$ is the ideal generated by all the invariants of positive degree i.e.,
$J_G=EE_+^G$, where
\[  E_+^G = \{ p \in E_+ \mid g \cdot p = p, \forall g \in G \} .\]
\end{Definition} 

Notice that  $E^G_+ = \bigoplus_{d>0} E_d^G$, where $E_d^G$ is the space of invariant skew polynomials of degree $d$. 

The key theorem in our method to compute invariant skew polynomials is an adaptation of Derksen's result to the exterior algebra context.

\begin{Theorem}\label{gansub}
Let $J_G$ be the Hilbert ideal for the action of $G$ on $E= \we(x_1, \ldots, x_n)$. Consider the ideal $I_G = \II'(\cA_G)$ in the ring $\we(x_1, \ldots, x_n, y_1, \ldots, y_n)$. We have that 
\[ (I_G + (y_1, \ldots, y_n) ) \cap \we(x_1, \ldots, x_n) = J_G. \]
\end{Theorem}

\begin{proof}
The easy containment is $(I_G + (y_1, \ldots, y_n) ) \cap \we(x_1, \ldots, x_n) \supseteq I_G $. Let $f(\x) \in I_G$ and rewrite $f(\x)$ as $f(\x) = (f(\x) - f(\y)) + f(\y)$. As $f(\y)$ is a skew polynomial of positive degree in the $y$ variables, we have that $f(\y) \in (y_1, \ldots, y_n)$. On the other hand, we have that $(f(\x) - f(\y)) \in J_G$ as for any $g \in G$, we have that
\[ \phi_g (f(\x)-f(\y)) = f(\x) - f(A(g)\x) = f(\x) - g \cdot f(\x) = 0, \]
because we assumed that $f$ in an invariant in $J_G$. This establishes the first containment.

Next, we will prove the second containment $(I_G + (y_1, \ldots, y_n) ) \cap \we(x_1, \ldots, x_n) \subseteq J_G $. Any $f(\x) \in \E$ that lies in $I_G + (y_1, \ldots, y_n)$ can be written as 
\[ f(\x) = p(\x,\y) + \sum c_i(\x) g_i(\y), \]
for $p(\x,\y) \in I_G$ and $\sum c_i(\x) g_i(\y) \in (y_1, \ldots, y_n)$. In particular, for each $i$ we require that $c_i(\x) \in \E$ and $g_i(\y)$ is a skew polynomial in the $y$ variables alone with no constant term.

Let $G$ act on $V \times V$ by the trivial action on the first copy of $V$, and the given action on the second copy of $V$. The Reynolds operator for this action is such that 
\[ \cR : \we(\x) \otimes \we(\y) \to (\we(\x) \otimes \we(\y))^G = \we(\x) \otimes \we(\y)^G .\]
In particular, $\cR$ is a $\we(\x)$-module homomorphism by property (ii) of the Reynolds operator. Notice that $\cR|_{\we(\x)} = \Id_{\we(\x)}$, whilst $\cR|_{\we(\y)}:  \we(\y) \to \we(\y)^G$ is the usual Reynolds operator on $\we(\y)$.
Apply $\cR$ to $f(\x)$. We get that
\[ \cR(f(\x)) = f(\x) = \cR(p(\x,\y)) + \sum \cR(c_i(\x)) \cR(g_i(\y)) = \cR(p(\x,\y)) + \sum c_i(\x) \cR(g_i(\y)). \]
Notice that $I_G$ is $G$-stable, so we have that $\cR(p(\x,\y)) \in I_G$, by the property (i) of the Reynolds operator. 

Consider the ring map 
\[ \delta: \we(\x,\y) \to \we(\x), \] 
given by $h(\x,\y) \mapsto h(\x,\x)$. Notice that this is the just the substitution $\phi_{\id}$ given by $\y = A(\id) = \id \x$. In particular, $\delta$ acts as the identity on skew polynomials in the subring $\E$. When we apply $\delta$ to $f(\x)$ we get
\[ \delta(f(\x)) = f(\x) = \delta\cR(p(\x,\y)) + \sum \delta(c_i(\x)) \delta\cR(g_i(\y)) = \delta\cR(p(\x,\y)) + \sum c_i(\x) \cR(g_i(\x)). \]
However, notice that any $h \in I_G$ must ``vanish" on the subspace $V_{\id}$ associated to the identity of $G$, meaning that $\phi_{\id}(h)=0$. Thus $h(\x,\x)=0$ for any $h \in I_G$. Hence, $I_G$ is in the kernel of $\delta$. In particular, $\delta\cR(p(\x,\y))=0$. Applying this observation to our expression for $f(\x)$, we conclude that:
\[ f(\x) = \sum c_i(\x) \cR(g_i(\x)). \]
As for any $g$ we have that $\cR(g)$ lies in the Hilbert ideal $J_G$, the above expression establishes that $f(\x) \in J_G$, as required. 
\end{proof}

\section{Noether's Degree Bound over the exterior algebra}
Next, we present our theorem providing a bound on the degree of the invariant skew polynomials in the exterior algebra. This is one of the main results of this paper.
\begin{Theorem}[Noether's bound for the exterior algebra]\label{nbe}
Let $\KK$ be a field of characteristic zero and $G$ a finite group acting on the finite dimensional vector space $V$. Then $\we(V^*)^G$ is generated in degree at most $|G|$.
\end{Theorem}

\begin{proof}
By subspace theorem for the exterior algebra Theorem \ref{subext}, we have that the ideal $J_G = \II'(\cA_G)$ is generated in degree at most $|G|$. Using Theorem \ref{gansub}, we get that the Hilbert ideal $I_G$ is given by
\[ (J_G + (y_1, \ldots, y_n) ) \cap \we(x_1, \ldots, x_n). \]
This means that generators of $I_G$ are obtained from generators of $J_G$ by setting the $y$-variables equal to $0$. Since $J_G$ is generated in degree $\leq |G|$, so is $I_G$.
 Let $\{f_1, \ldots, f_r\}$ be a set of generators for $J_G$, then by Hilbert invariant theorem (Theorem \ref{hit}) we have that $\we(\x)^G = \langle \cR(f_1), \ldots, \cR(f_r)\rangle$. Notice that the Reynolds operator does not increase the degree of a skew-polynomial. Therefore,  $\we(\x)^G$ is generated in degree at most $|G|$, as required. 
\end{proof}

We can restate the above theorem by saying that in characteristic zero for $G$ a finite group acting on $\we(V^*)$, we have that $\beta_V(G) \leq |G|$ for all $V$ and $G$, so we have the Noether's Degree Bound $\beta(G) \leq |G|$ holds in this setting.

In general, Noether's Degree Bound does not hold in the non-commutative setting. Consider for example the ring
\[ F  = \frac{\KK \langle x_1, \ldots, x_n\rangle}{(x_ix_j+x_jx_i, 1\leq i< j\leq n)}, \]
the skew  polynomial ring in $x_1, \ldots, x_n$. Notice that the exterior algebra can be obtained as a quotient of $F$, namely $E = \frac{F}{(x_i^2, 1\leq i\leq n)}$. Recent work of Kirkman, Kuzmanovich, and Zhang \cite{kirk} shows that Noether's Degree Bound does not hold for a finite group $G$ acting on $F$. In particular, they show that the group $\ZZ/(2)$ has a minimal invariant of degree 3.

\begin{Example}[Example 3.1 in \cite{kirk}]
Consider the permutation representation of $G = \ZZ/ (2)$ on $\KK\langle x,y\rangle/(xy+yx)$. This means that the generator $g$ of $G$ acts by swapping the variables $x$ and $y$. We have a linear invariant $f_1 = x+y$. However, the quadratic invariant $f_2 = x^2 +y^2$ is not a minimal invariant as
\[ f_1^2 = (x+y)^2 = x^2 + xy + yx + y^2 = x^2 + y^2 = f_2, \]
by the defining equations of the $(-1)$-skew polynomial ring $F$. The next minimal invariant is the cubic invariant $f_3 = x^3 +y^3$. One can show that the set $\{ f_1, f_3 \}$ is a set of minimal generating invariants. Thus, for this representation $V$ of $G$ we have that $\beta_V(G) = 3 > |G| = 2$.
\end{Example}

Consider the  ideal of squares $I = (x_i^2  \, , \, 1\leq i \leq n)$. One can notice that the exterior algebra is a quotient of the ring $F$ and more precisely, $E = F/I$. Our results on invariant skew polynomials in $E$ does has consequences for invariant skew polynomials in $F$.

\begin{Corollary}
Suppose that $I = (x_i^2  \, , \, 1\leq i \leq n)$ is a $G$-stable ideal. If $f$ is an invariant skew polynomial for the action of $G$ on $F$ and degree $\deg(f) > |G|$, then $f$ is not square-free.
\end{Corollary}

\begin{proof}
Towards contradiction, suppose that $f$ is a square-free skew invariant polynomial of degree $\deg(f) > |G|$. Consider $\overline{f} = f + I $ in $E=F/I$. Notice that $f \notin I$, because $f$ is square-free, so $\overline{f} \neq 0 $ and $\deg(\overline{f}) = \deg(f) > |G|$. Because $I$ is $G$-stable, we have that $\overline{f}$ is an invariant skew polynomial over $E$. By \ref{nbe}  $ \deg(f) =\deg(\overline{f})  \leq |G|$, contradicting the initial assumption.
\end{proof}

The exterior algebra case is special among non-commutative algebras. On the other hand, it also has features different from the ones of the symmetric algebra. In characteristic zero, given a finite group $G$ acting on an $n$-dimensional vector space $V$, the associated action on the polynomial ring $\R$ is such that $\beta_{V}(G) \leq \beta_{V_{\reg}}(G)$, where $V_{\reg}$ is the regular representation of $G$ \cite{schmid}. One says that the degree bound is achieved by the regular representation. However, this behavior does not carry to $\E$.

\begin{Example}
Consider the group $G = \ZZ/(2)$ and let $g$ generate $G$. Consider the action of $g$ on $\KK[x,y]$ given by $g \cdot x = x$ and $g \cdot y =  -y$. The polynomial ring $\KK[x,y]$ with this action is equivalent to the regular representation of $G$. In fact, we have a degree two invariant, $y^2$, which achieves Noether's Degree Bound so that $\beta(G)=2$.

Consider now the same action on the variables $x,y$, but in the exterior algebra $\we(x,y)$. We only have linear invariants, specifically the non-zero constant multiples of $x$. If the regular representation did achieve the Degree Bound, we would have that $\beta(G) \leq 1$. However, consider now two copies of the same representation. The action of $g$ on the variables $x_1,y_1,x_2,y_2$ is given by $g \cdot x_i = x_i$ and $g \cdot y_i = -y_i$, for $i =1,2$. In $\we(x_1,x_2,y_1,y_2)$ we now have a quadratic invariant: $y_1 \wedge y_2$. Thus, Noether's Degree Bound is achieved and we can conclude that $\beta(G)=2$.
\end{Example}

\section{Bound transference and Weyl's Polarization Theorem}

Another result in classical invariant theory that does not hold over $\E$ is Weyl's Polarization Theorem. Classical references on this result are \cites{w} and \cites{kp}.

\begin{Theorem}[Weyl's Polarization Theorem - weak form  \cite{dm}]
Assume that the characteristic of the ground field is zero and let $m = \dim W$. Then for any finite group $G$ acting on $W$, we have that $\beta_{W^m}(G) \geq \beta_{W^b}(G)$ for any other $b$.
\end{Theorem}

However, this result hold in the exterior algebra. In fact, it is not true that the highest degree minimal invariants appear in $\we(W^m) = \we(W \otimes \KK^m)$, the exterior algebra over the direct sum of $m$ copies of the representation $W$. We illustrate this in the following example. 

\begin{Example}
Let $W$ be the one-dimensional representation of $G = \ZZ/(2) = \langle g \rangle$ given by $g \cdot x = -x$. Then the only invariants in $\we(x)$ are scalars. On the other hand, consider the representation $W^2$, where $g \cdot x_i = x_i$, for $i=1,2$. We do have maximal degree invariants now: we have the quadratic invariant $x_1 \wedge x_2$. 
\end{Example}

Over $\sym$, Weyl's Polarization Theorem works because of the decomposition
\[\sym(W\otimes \KK^n) = \bigoplus \Sl(W) \otimes \Sl(\KK^n). \]


For any $W$ such that $m = \dim(W) \geq l(\lambda)$, we have that $\Sl(W) \neq 0$.
Similarly, when $n \geq l(\lambda)$, we have that we have that $\Sl(\KK^n) \neq 0$. Assuming that $n \geq m$, the decomposition becomes
\[\sym(W\otimes \KK^n) = \bigoplus_{\lambda,  m \geq l(\lambda)   } \Sl(W) \otimes \Sl(\KK^n). \]
From the inclusion map $\KK^{n-1} \into \KK^n$, we get the map
\[ \phi: \sym(W\otimes \KK^{n-1}) \into \sym(W\otimes \KK^n).\]
Considering $W$ as a representation of $G$ and letting $G$ act trivially in $\KK^n$ we get the diagram
\begin{center}

\begin{tikzcd}
 \sym(W\otimes \KK^n)^G \arrow[r, "\cong"] & \bigoplus_{\lambda,  l(\lambda) \leq m } \Sl(W)^G \otimes \Sl(\KK^n) \\
  \sym(W\otimes \KK^{n-1})^G \arrow[r, "\cong"] \arrow[u,"\phi"] & \bigoplus_{\lambda, l(\lambda) \leq m  } \Sl(W)^G \otimes \Sl(\KK^{n-1}) \arrow[u, "\phi"] 
\end{tikzcd} 
    
\end{center}
As an equivariant map $\phi$ maps the component indexed by $\lambda$ into the component indexed by $\la$. So as long as the image of $\phi$ on a $\la$-component is non-zero, by Schur's Lemma the image is actually the whole $\la$-component in the target space.

On the other hand, for $\we$ we have the decomposition 
\[\we(W\otimes \KK^n) = \bigoplus \Sl(W) \otimes \Slp(\KK^n). \]
with the presence of partitions $\la$ and $\la'$.  In particular, a map 
\[ \Sl(W) \otimes \Slp(\KK^{n-1}) \to \Sl(W) \otimes \Slp(\KK^{n})\] 
might be zero because $m < |\la| $ or $n-1 < |\la'|$. But a restriction on the size of $\la'$ does put any restriction on the size of $\la$ (only on its first row). The asymmetry makes it impossible to simultaneously characterize which components are non-zero.

Even though Weyl's Polarization Theorem does not hold over the exterior algebra, we can still find a bound on the number of copies of a vector space required to achieve maximal degree invariants. We will use the transpose functor $\W$ to transfer bounds from the symmetric algebra to the exterior algebra.

Another way to say that there is an upper bound $d$ on the degree of the minimal invariants over the polynomial algebra $\sym$ is to say that the map $m$ 
\[m: \bigoplus_{e=1}^{d-1} \sym_e^G \otimes \sym_{d-e}^G \to \sym_d^G \]
is surjective. When $m$ is surjective, there are no minimal invariants in degree $d$ because all invariants in degree $d$ can be obtained polynomials in smaller degree invariants. 

Now consider the action of the finite group $G$ on a vector space $W$. We consider the polynomial functors $\Sl(\blank \otimes W)$ that for any vector space $V$, maps $ V \to \Sl(V \otimes W)$. As in ~\cites{thesis,article}, we define $\W(\Sl)$ to be the direct limit of $ V \to \Slp(V \otimes W)$, as $\dim(V) \to \infty$, which is naturally equivalent to $\Slp(\blank \otimes W)$. Notice that because $G$ acts on the fixed vector spaces $W$, the group $G$ acts trivially on $\Sl(V)$, for any $V$, and thus the action of $G$ commutes with $\W$, as in the diagram below.
\begin{center}
  \begin{tikzcd}
 \sym(\blank \otimes W) \arrow[r, "\W"] & \we(\blank \otimes W)   \\
  \sym(\blank \otimes W)^G  \arrow[hookrightarrow]{u} \arrow[d,"\cong"] \arrow[r, "\W"] & \we(\blank \otimes W)^G \arrow[d,"\cong"] \arrow[hookrightarrow]{u} \\
   \bigoplus \Sl(\blank) \otimes \Sl(W)^G  \arrow[r, "\W"] & \bigoplus \Slp(\blank) \otimes \Sl(W)^G 
\end{tikzcd}  
\end{center}

We now combine what we know about the map $m$ over $\sym$ with the observation about the action of $G$ to obtain the following theorem.

\begin{Theorem}
Consider the action of $G$ on the vector space $W$ and let $V$ be any vector space with a trivial action of $G$. Suppose that the maximal degree of a minimal invariant in $\sym( W\otimes V)^G$ is less than $d$, so  $\beta_{\wv}(G) < d$ for any vector space $V$. Then the maximal degree of a minimal invariant over $\we(\wv)^G$ is also less than $d$, so in symbols $\beta_{\wv}'(G) < d$.
\end{Theorem}

\begin{proof}
Because $\beta_{\wv}(G) < d$ for any vector space $V$, we have that the map $m$ 
\[m: \bigoplus_{e=1}^{d-1} \sym(\wv)_e^G \otimes \sym(\wv)_{d-e}^G \to \sym(\wv)_d^G \]
is surjective. 
Notice that for any integer $a$, we have that 
\[\sym(\wv)_a^G = \bigoplus_{\lambda \vdash a} \Sl(V)^G \otimes \Sl(W)^G =  \bigoplus_{\lambda \vdash a} \Sl(V) \otimes \Sl(W)^G, \] 
because $G$ acts trivially on $V$ and thus $G$ also acts trivially on $\Sl(V)$. Using this observation for the map $m$ we obtain that
\[m: \bigoplus_{e=1}^{d-1} \left(\bigoplus_{\lambda \vdash e} \Sl(V) \otimes \Sl(W)^G \right) \otimes \left(\bigoplus_{\mu \vdash d-e} \Sm(V) \otimes \Sm(W)^G \right) \to \bigoplus_{\nu \vdash d} \Sn(V) \otimes \Sn(W)^G. \]
Applying the functor $\W_V$ to this map we get, for any vector space $V$, the map
\[\W_V(m): \bigoplus_{e=1}^{d-1} \left(\bigoplus_{\lambda \vdash e} \Slp(V) \otimes \Sl(W)^G \right) \otimes \left(\bigoplus_{\mu \vdash d-e} \Smp(V) \otimes \Sm(W)^G \right) \to \bigoplus_{\nu \vdash d} \Snp(V) \otimes \Sn(W)^G. \]
Considering the $\lambda'$-component, an equivariant map is nonzero only when $\lambda' \subset \nu'$. So it is impossible for $\Snp(V) \neq 0$, but $\Slp(V) = 0$. So for $\dim(V) < d$, the map $\W_V(m)$ might be the zero map, but it is nonetheless surjective because when $\Slp(V) = 0$ we also have $\Snp(V) = 0$. 

Because we have a surjective map for any $V$, the map $\W(m) = m'$ is surjective. Therefore, 
\[m': \bigoplus_{e=1}^{d-1} \we(\wv)_e^G \otimes \we(\wv)_{d-e}^G \to \we(\wv)_d^G \]
is surjective and so $\beta_{\wv}'(G) < d$.
\end{proof}

The above results establishes a similar result to Weyl's Polarization Theorem, but with a new bound: instead of using the dimension of the representation $W$, we use the order of the group. Next, we will consider a the stronger version of Weyl's Polarization Theorem and prove a corollary with the same flavor in this new context.

Let $E$ be a $\GL(V)$ representation. For any subset $ S \subseteq E$, we define $\langle S \rangle_{\GL(V)}$ to be the smallest $\GL(V)$-stable subspace containing $S$. For $r \geq s$, the inclusion $W^s \subseteq W^r$ induces an inclusion $\GL_s \into \GL_r$ given by
\[A \mapsto \begin{bmatrix} A & 0 \\
0 & I_{r-s} 
\end{bmatrix} := \widetilde{A}. \]
The action of $\GL_r$ on a $\GL_s$-module via this inclusion is referred to as the operation of polarization. Concretely, for $S \subseteq E$ and $E$ a $GL_s$-module, we have that 
\[\langle S \rangle_{\GL_r} = \left\{ \sum_{f\in S \, , \,  A \in \GL_r} \widetilde{A} \cdot f  \right\}. \] 
Weyl's Polarization Theorem is rephrased as follows in terms of this construction.

\begin{Theorem}[Weyl's Polarization Theorem - strong form  \cite{dm}]\label{wps}
Assume that the characteristic of the ground field $\KK$ is zero and let $m = \dim W$.  Let $r \geq s \geq m$ and let $S \subset \KK[W^s]^G$ be a generating set for $\KK[W^s]^G$. Then the set
$\langle S \rangle_{\GL_r} $ is a generating set for $\KK[W^r]^G$.
\end{Theorem}

\begin{Remark}
In characteristic zero, when $\Sl(W^s)$ is non-zero and $r \geq s$, the $\GL_r$-module $\langle \Sl(W^s) \rangle_{\GL_r} \subset \Sl(W^r)$ is also non-zero. By Schur's Lemma, $\langle \Sl(W^s) \rangle_{\GL_r} = \Sl(W^r)$ because $\Sl(W^r)$ is an irreducible $\GL(W^r)$-representation. This is the key idea in the proof of \ref{wps} in \cites{dm}.
\end{Remark}

Given that Noether's Degree Bound does hold in the exterior algebra, we prove the following polarization bound for the exterior algebra.

\begin{Corollary}
Assume that the characteristic of the ground field $\KK$ is zero and let $t = |G|$.  Let $r \geq s \geq t$ and let $S \subset \we(W^s)^G$ be a generating set for $\we(W^s)^G$. Then the set
$\langle S \rangle_{\GL_r} $ is a generating set for $\we(W^r)^G$.
\end{Corollary}

\begin{proof}
Consider $\we(W^s)_d$, the space of degree $d$ forms in $\we(W^s)$. Recall that for $V \cong \KK^s$ 
\[ \we(W^s)_d = \bigoplus_{\lambda \vdash d} \Sl(W) \otimes \Slp(V). \]
If $\Slp(V)$ is non-zero, then for $r\geq s$ and $V'\cong \KK^r$, we also have that $\Slp(V')$ is non-zero. Thus for a non-zero component $\Sl(W) \otimes \Slp(V)$ injecting into $\Sl(W) \otimes \Slp(V')$, we can conclude that $\langle \Sl(W) \otimes \Slp(V) \rangle_{\GL_r} = \Sl(W) \otimes \Slp(V')$.

Then, for $d \leq s \leq r$ such that $\we(W^s)_d$ is non-zero, we have that $ \langle \we(W^s)_d \rangle_{\GL_r} = \we(W^r)_d$ and by considering the $G$-stable component, we have that \[\langle \we(W^s)^G_d \rangle_{\GL_r} = \we(W^r)^G_d . \]

By Noether's Degree Bound for the exterior algebra, we know that a set of minimal generating invariants $S$ for $\we(W^s)^G$ can be found in degrees less than equal to $t = |G|$. Because $s \geq t$, we have that $S \subset \sum_{d=0}^s \we(W^s)_d^G $. Then for $r \geq s$, every invariant up to degree $s$ in $\we(W^r)^G$ can be obtained by acting with $\GL_r$ on $S$. However, a minimal generating set of invariants for $\we(W^r)^G$ can be found in degrees $\leq t \leq s$. Therefore, $\langle S \rangle_{\GL_r} $ is a generating set for $\we(W^r)^G$.

\end{proof}







Informally, the above result says that one can ``see" all the invariants- and in particular the highest degree minimal invariants -in $|G|$-many copies of the representation $W$. So for $r > t= |G|$, we can find all invariants for the action of $G$ on $\we(W^{r})$ by considering the action of $G$ on $\we(W^t)$ and using polarization. Therefore, we only need as many copies of $W$ as there are elements in the group $G$ to be guaranteed to observe highest degree invariants.

\end{document}